\documentclass[leqno,12pt]{amsart}
\usepackage{setspace}
\usepackage[english]{babel}
\usepackage{bussproofs}
\usepackage{dirtytalk}
\usepackage{mathtools}
\usepackage{mathrsfs}
\usepackage{latexsym}
\usepackage{enumitem}
\usepackage{amsthm}
\usepackage{amssymb}
\DeclareMathAlphabet{\mathpzc}{OT1}{pzc}{m}{it}
\usepackage[latin1]{inputenc}
\usepackage{afterpage}

\usepackage{bbm}
\usepackage[rgb,dvipsnames]{xcolor}
\usepackage{tikz} 
\usetikzlibrary{decorations.text}
\usetikzlibrary{arrows,arrows.meta,hobby}
\usetikzlibrary{shapes.misc, positioning}

\usepackage{tikz-cd}

\usetikzlibrary{cd}

\newtheorem{theorem}{Theorem}[section]

\newtheorem{lemma}[theorem]{Lemma}
\newtheorem{proposition}[theorem]{Proposition}
\newtheorem{corollary}[theorem]{Corollary}

\newtheorem{fact}[theorem]{Fact}
\newtheorem{claim}[theorem]{Claim}
\theoremstyle{definition}
\newtheorem{definition}[theorem]{Definition}

\theoremstyle{remark}
\newtheorem{remark}{Remark}

\newtheorem{question}{Question}

\def\hook{\upharpoonright}
\def\forces{\Vdash}

\def\Me{\mathcal M}
\def\Null{\mathcal N}
\def\ZFC{\mathsf{ZFC}}

\def\Kb{\mathcal K}
\def\baire{\omega^\omega}
\def\bbb{(\omega^\omega)^{\omega^\omega}}

\def\mfc{\mathfrak{c}}
\def\mfb{\mathfrak b}
\def \mfd{\mathfrak{d}}
\def\GCH {\mathsf{GCH}}
\def\CH {\mathsf{CH}}

\usepackage[margin=1.3in]{geometry}

\title{Higher Dimensional Cardinal Characteristics for Sets of Functions}

\author[Switzer]{Corey Bacal Switzer}
\address[C.~B.~Switzer]{Institut f\"{u}r Mathematik, Kurt G\"odel Research Center, Universit\"{a}t Wien, Kolingasse 14-16, 1090 Wien, AUSTRIA}
\email{corey.bacal.switzer@univie.ac.at}

\thanks{\emph{Acknowledgements:} The author would like to thank the
Austrian Science Fund (FWF) for the generous support through grant number Y1012-N35.}

\date{}


\date{}

\begin{document}
\maketitle

\begin{abstract}
Much recent work in cardinal characteristics has focused on generalizing results about $\omega$ to uncountable cardinals by studying analogues of classical cardinal characteristics on the generalized Baire and Cantor spaces $\kappa^\kappa$ and $2^\kappa$. In this note I look at generalizations to other function spaces, focusing particularly on the space of functions $f:\omega^\omega \to \omega^\omega$. By considering classical cardinal invariants on Baire space in this setting I derive a number of \say{higher dimensional} analogues of such cardinals, ultimately introducing 18 new cardinal invariants, alongside a framework that allows for numerous others. These 18 form two separate diagrams consisting of 6 and 12 cardinals respectively, each resembling versions of the Cicho\'n diagram. These $\ZFC$-inequalities are the first main result of the paper. I then consider other relations between these cardinals, as well as the cardinal $\mfc^+$ and show that these results rely on additional assumptions about cardinal characteristics on $\omega$. Finally, using variations of Cohen, Hechler, and localization forcing I prove a number of consistency results for possible values of the new cardinals.
\end{abstract}

\section{Introduction}
Many cardinal characteristics on $\omega^\omega$ arise as follows: fix some relation $R \subseteq \omega \times \omega$ and let $R^* \subseteq \omega^\omega \times \omega^\omega$ be defined by $fR^*g$ if and only if for all but finitely many $n$ $f(n) R g(n)$. For instance, letting $R$ be the the usual order on $\omega$ gives the eventual domination ordering. Each such $R$ then gives rise to two cardinal characteristics, $\mathfrak{b}(R^*)$, the least size of a set $A \subseteq \omega^\omega$ with no $R^*$-bound and $\mathfrak{d}(R^*)$ the least size of a set $D \subseteq \omega^\omega$ which is $R^*$-dominating. A natural generalization of this as is follows: fix two sets $X$ and $Y$, let $\mathcal I$ be an ideal on $X$ and $R \subseteq Y \times Y$ be a binary relation on $Y$. Let $Y^X$ be the set of functions $f:X \to Y$ and consider the relation $R_\mathcal I \subseteq Y^X \times Y^X$ given by $f R_\mathcal I g$ if and only if for almost all $x$ we have $f(x) R g(x)$ i.e. $\{x \in X \; | \; \neg f(x) R g(x)\} \in \mathcal I$. Again we get two cardinal characteristics, this time on the set $Y^X$: $\mathfrak{b}(R_\mathcal I)$, the least size of a set $A \subseteq Y^X$ which has no $R_\mathcal I$-bound and $\mathfrak{d}(R_\mathcal I)$, the least size of a set $D \subseteq Y^X$ which is dominating with respect to $R_\mathcal I$. Note that letting $X=Y = \omega$ and $\mathcal I$ be the ideal of finite sets we recover the original setting for cardinal characteristics on Baire space and letting $Y = 2$ we recover the same for Cantor space.

Recently, much work has been done on the case of $X = \kappa$ and $Y = \kappa$ or $2$ for arbitrary $\kappa$, thus generalizing cardinal characteristics to larger cardinals, see for example the article \cite{brendle16} or the survey \cite{questionsbaire} for a list of open questions. In this case the interesting ideals are the ideal of sets of size $< \kappa$, the non-stationary ideal, and, if $\kappa$ has a large cardinal property, then potentially some ideal related to this. See \cite{CuSh95}, Theorems 6 and 8 for a particularly striking result relating cardinal invariants modulo different ideals. 

However, this framework is more flexible than just allowing one to study generalized Baire space and Cantor space. Indeed it is easy to imagine numerous new cardinal characteristics. In this article I consider a different generalization, based on the function space $(\omega^\omega)^{\omega^\omega}$ of functions $f:\omega^\omega \to \omega^\omega$. Since Baire space comes with ideals that are not easily defined on $\kappa^\kappa$ we get further generalizations of cardinal characteristics. Specifically I will consider the ideals $\mathcal N$, $\mathcal M$ and $\mathcal K$ of null, meager and (contained in) $\sigma$-compact sets respectively. The result is a \say{higher dimensional} version of several well-known cardinal characteristics. While many different generalizations are possible let me stick with the following three relations for simplicity: $\leq^*$, the relation of eventual domination, $\neq^*$, the relation of eventual non-equality and $\in^*$ the relation of eventual capture (these relations will be defined below). By considering two cardinals for each of these three relations and three ideals I end up with 18 new cardinals characteristics above the continuum. The first main theorem of this article is to show that these \say{higher dimensional} cardinals behave, provably under $\ZFC$, similar to their Baire space analogues (the cardinals mentioned below will be defined in detail in the next section).

\begin{theorem}
Interpreting $\to$ as $\leq$ the inequalities shown in Figures 1 and 2 are all provable in $\ZFC$.

\begin{figure}[h]
\centering
  \begin{tikzpicture}[scale=1.5,xscale=2]
     \draw 
           (1,0) node (Bin*) {$\mfb(\in_\Null^*)$}
           (1,1) node (Bleq*) {$\mfb(\leq_\Null^*)$}
           (1,2) node (Bneq*) {$\mfb(\neq_\Null^*)$}
           (2,0) node (Dneq*) {$\mfd(\neq_\Null^*)$}
           (2,1) node (Dleq*) {$\mfd(\leq_\Null^*)$}
           (2,2) node (Din*) {$\mfd(\in_\Null^*)$}
           ;
     \draw[->,>=stealth]
            (Bin*) edge (Bleq*)
            (Bleq*) edge (Bneq*)
            (Bin*) edge (Dneq*)
            (Bleq*) edge (Dleq*)
            (Bneq*) edge (Din*)
            (Dneq*) edge (Dleq*)
            (Dleq*) edge (Din*)
;      
  \end{tikzpicture}
\caption{Higher Dimensional Cardinal Characteristics Mod the Null Ideal}
\end{figure}

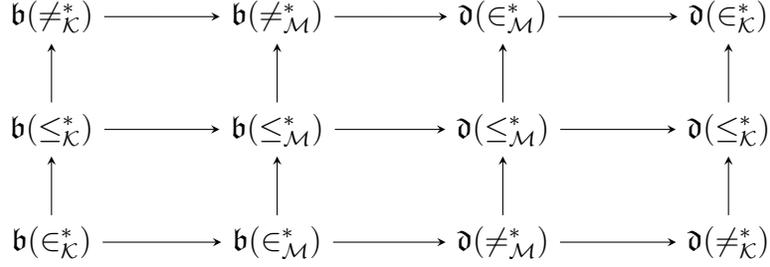
\begin{figure}[h]
\centering
  \begin{tikzpicture}[scale=1.5,xscale=2]
     \draw 
           (1,0) node (Bin*) {$\mfb(\in_\Me^*)$}
           (1,1) node (Bleq*) {$\mfb(\leq_\Me^*)$}
           (1,2) node (Bneq*) {$\mfb(\neq_\Me^*)$}
           (2,0) node (Dneq*) {$\mfd(\neq_\Me^*)$}
           (2,1) node (Dleq*) {$\mfd(\leq_\Me^*)$}
           (2,2) node (Din*) {$\mfd(\in_\Me^*)$}
	     (0,0) node (BinK*) {$\mfb(\in_\Kb^*)$}
           (0,1) node (BleqK*) {$\mfb(\leq_\Kb^*)$}
           (0,2) node (BneqK*) {$\mfb(\neq_\Kb^*)$}
           (3,0) node (DneqK*) {$\mfd(\neq_\Kb^*)$}
           (3,1) node (DleqK*) {$\mfd(\leq_\Kb^*)$}
           (3,2) node (DinK*) {$\mfd(\in_\Kb^*)$}
           ;
     \draw[->,>=stealth]
            (Bin*) edge (Bleq*)
            (Bleq*) edge (Bneq*)
            (Bin*) edge (Dneq*)
            (Bleq*) edge (Dleq*)
            (Bneq*) edge (Din*)
            (Dneq*) edge (Dleq*)
            (Dleq*) edge (Din*)
            (Dleq*) edge (DleqK*)
            (Dneq*) edge (DneqK*)
            (Din*) edge (DinK*)
            (BleqK*) edge (Bleq*)
            (BneqK*) edge (Bneq*)
            (BinK*) edge (Bin*)
            (BinK*) edge (BleqK*)
            (BleqK*) edge (BneqK*)
            (DneqK*) edge (DleqK*)
	      (DleqK*) edge (DinK*)

;
 \end{tikzpicture}
\caption{Higher Dimensional Cardinal Characteristics Mod the Meager and $\sigma$-Compact Ideals}
\end{figure}

\end{theorem}

\pagebreak

By contrast, the relation between these cardinals and the size of the continuum is less clear. To show that these cardinals are at least $\mfc^+$, I need to assume certain values of cardinal characteristics. 
\begin{theorem}
Assume that $add (\Null) = \mfc$, then all of the cardinals in Figures 1 and 2 are at least $\mfc^+$.
\end{theorem}
Similarly, cardinal characteristics on $\omega$ play a role in relating the cardinals modulo the different ideals. Specifically, assuming all the cardinals in Cicho\'n's diagram are equal, I prove that the cardinals obtained by modding out by the null and meager ideals respectively are equal. This result is thematically similar to the aforementioned theorems from \cite{CuSh95} relating $[\kappa]^{< \kappa}$ to $\mathsf{NS}_\kappa$.
\begin{theorem}
Assume $cof(\Null) = add (\Null)$. Then for each $R \in \{\in^*, \leq^*, \neq^*\}$ we have that $\mfb (R_\Null) = \mfb (R_\Me)$ and $\mfd (R_\Null) = \mfd (R_\Me)$.
\end{theorem}

Next I turn to consistency results and show the following.
\begin{theorem}
Let $\aleph_2 \leq \kappa \leq \lambda$ with $\kappa$ and $\lambda$ regular and $\mathcal I \in  \{\Null, \Me, \Kb\}$. Each of the following inequalities are consistent.
\begin{enumerate}
\item
$\mfc^+ = \mfb (\neq^*_\mathcal I) < \mfd (\neq^*_\mathcal I) = \kappa = 2^\mfc$.
\item
$\mfc^+ < \mfb (\leq^*_\mathcal I) = \mfd (\neq^*_\mathcal I) = \kappa = 2^\mfc$.
\item
$\mfc^+ < \mfb (\leq^*_\mathcal I) = \kappa < \mfd (\leq^*_\mathcal  I) = \lambda = 2^\mfc$.
\item
$\mfc^+ < \mfb (\in^*_\mathcal I) = \kappa = 2^\mfc$.
\end{enumerate}
\end{theorem}

Many other consistency results seem within grasp using standard techniques or minor modifications and I list some open problems related to these at the end of the paper.

In the rest of the introduction I introduce terminology that will be used throughout. In the next section the cardinals $\mathfrak{b}(R_\mathcal I)$ and $\mathfrak{d}(R_\mathcal I)$ are introduced and basic relations between them are shown. The third section investigates the relation between these higher dimensional cardinal characteristics and the standard cardinal characteristics on $\omega$. Section 4 contains a number of consistency results and introduces three new forcing notions based on generalizations of Cohen, Hechler, and localization forcing. In section 5 I list a number of open questions, as well as some possible extensions. 

In what follows I use letters like $x, y, z$ to denote elements of Baire space and letters like $f, g, h$ to denote functions from $\omega^\omega$ to $\omega^\omega$ (or, between uncountable Polish spaces more generally). If $\mathcal I$ is an ideal, a set is $\mathcal I$-positive if it's not in $\mathcal I$ and is $\mathcal I$-measure one if its complement is in $\mathcal I$. Note that in the case of the null ideal, non measurable sets are $\mathcal I$-positive. Otherwise the notation is standard conforming to that \cite{KenST}. Also, I use the monograph \cite{BarJu95} as the standard reference for cardinal characteristics of the continuum and occasionally refer to the survey article \cite{BlassHB} as well.

Let me recall the following vocabulary. Given any set $X$ and a relation $R$ on $X$ an element $x \in X$ is an $R$-{\em bound} for a set $A \subseteq X$ if for every $a \in A$ we have that $a R x$. A set is $R$-{\em bounded} if it has an $R$-bound. It's $R$-{\em unbounded} otherwise. A set $D \subseteq X$ is $R$-{\em dominating} if for every $y \in X$ there is a $d \in D$ so that $y R d$. For any such $X$ and $R$ I write $\mfb (R)$ for the least size of an $R$-unbounded set and $\mfd(R)$ for the least size of an $R$-dominating set. If $\mathbb{Q} = (Q, \leq_\mathbb Q)$ is a partially ordered set then I also write $\mfb (\mathbb Q)$ and $\mfd (\mathbb Q)$ for $\mfb (\leq_\mathbb Q)$ and $\mfd(\leq_\mathbb Q)$ respectively.

I let $\mu$ denote the Lebesgue measure on $\omega^\omega$ (or any other oft-encountered Polish space under consideration). Recall that this is defined as follows. For a finite sequence $s\in \omega^{<\omega}$ let $N_s = \{ x \in \baire \; | \; s \subseteq x\}$ be the basic open determined by $s$ and let $\mu(N_s) := \Pi_{i < {\rm length(s)}} 2^{-s(i) + 1}$. The measure $\mu$ is then extended to all measurable sets in the normal way.

The symbols $\Null$, $\Me$, $\Kb$ denote the null ideal, the meager ideal and the ideal generated by $\sigma$-compact subsets of $\baire$ respectively. In the case of $\Kb$ recall that it is a well known result that $A \in \Kb$ if and only if $A$ is $\leq^*$-bounded, see the proof of Theorem 2.8 of \cite{BlassHB}. The relevant properties that all three of these ideals share is that they are non-trivial $\sigma$-ideals containing all countable subsets of $\baire$ and have a Borel base: every element of each ideal is covered by a Borel set in that ideal. In the case $\Null$ and $\Me$ the fact that the underlying set is $\baire$, as opposed to any other perfect Polish space is unimportant in this paper, however, it obviously matters for $\Kb$ since many Polish spaces are themselves $\sigma$-compact and hence $\Kb$ on such a space is trivial.

Recall that a function $s:\omega \to [\omega]^{< \omega}$ is called a {\em slalom} if $|s(n)| \leq n$ for all $n$. The set of all slaloms is denoted $\mathcal S$. Given the product topology of the discrete topology on $\omega^{<\omega}$ the space $\mathcal S$ is homeomorphic to Baire space. Given an element $x \in \omega^\omega$ let $x \in^* s$ if for all but finitely many $n$ $x(n) \in s(n)$. Using the terminology of the first paragraph, recall the following well known theorem, due to Bartoszy\'nski.

\begin{fact}[Bartoszy\'nski, see Theorem 2.3.9 of \cite{BarJu95}]

The following equalities are provable in $\ZFC$.

\begin{enumerate}

\item
$\mathfrak{b} (\in^*) = add(\Null)$
\item
$\mathfrak{d}(\in^*) = cof (\Null)$
\end{enumerate}
\end{fact}

A similar theorem due to Miller relates the meager ideal to the relation $\neq^*$.

\begin{fact}[Miller \cite{Miller81}, see Thereoms 2.4.1 and 2.4.7 of \cite{BarJu95}]
The following equalities are provable in $\ZFC$.
\begin{enumerate}
\item
$\mfb (\neq^*) = non (\Me)$
\item
$\mfd (\neq^*) = cov (\Me)$
\end{enumerate}
\end{fact}

\section{Higher Dimensional Variants of A Fragment of Cicho\'n's Diagram}

In this section I define the cardinals that will be studied for the rest of the paper. The basic definition is given below.

\begin{definition}
Let $\mathcal I \in \{\mathcal N, \mathcal M, \mathcal K\}$ and $R \in \{\leq^*, \neq^*, \in^*\}$.
\begin{enumerate}
\item
$\mathfrak{b}(R_\mathcal I)$ is the least size of a set $A$ of functions from $\omega^\omega$ to $\omega^\omega$ for which there is no $g:\omega^\omega \to \omega^\omega$ ($g:\omega^\omega \to \mathcal S$ in the case of $R = \in^*$) such that for all $f \in A$ the set $\{x \in \omega^\omega \; | \; \neg f(x) R g(x)\}$ is in $\mathcal I$.
\item
$\mathfrak{d}(R_\mathcal I)$ is the least size of a set $A$ of functions from $\omega^\omega$ to $\omega^\omega$ ($\baire$ to $\mathcal S$ in the case of $R = \in^*$) so that for all $g:\omega^\omega \to \omega^\omega$ there is an $f \in A$ for which the set $\{x \in \omega^\omega \; | \; \neg g(x) R f(x)\}$ is in $\mathcal I$.
\end{enumerate}
\end{definition}

By varying $\mathcal I$ and $R$ this definition gives 18 new cardinals. For readability, let me give the details below for the case of the null ideal. Similar statements hold for $\Me$ and $\Kb$. First let's see explicitly what each relation $R_\mathcal I$ is. On the two lists below let $f, g:\baire \to \baire$ and $h:\baire \to \mathcal S$.

\begin{enumerate}
\item
$f \neq^*_\Null g$ if and only if for all but a measure zero set of $x \in \baire$ we have that $f (x) \neq^* g(x)$.

\item
$f \leq^*_\Null g$ if and only if for all but a measure zero set of $x \in \baire$ we have that $f(x) \leq^* g(x)$.

\item
$f \in^*_\Null h$ if and only if for all but a measure zero set of $x \in \baire$ we have that $f(x) \in^* h(x)$.
\end{enumerate}

For the cardinals now we get the following. Note that $\neg x \neq^* y$ means $\exists^\infty n \, x(n) = y(n)$ and the same for the other relations.

\begin{enumerate}

\item
$\mfb(\neq^*_\Null)$ is the least size of a $\neq^*_\Null$-unbounded set $A \subseteq \bbb$ i.e. $A$ is such that for each $f:\baire \to \baire$ there is a $g \in A$ so that the set of $\{x \; | \; \exists^\infty n \, g(x)(n) = f(x)(n)\}$ is not measure zero.

\item
$\mfd(\neq^*_\Null)$ is the least size of a $\neq^*_\Null$-dominating set $A \subseteq \bbb$ i.e. $A$ is such that for every $f:\baire \to \baire$ there is a $g \in A$ so that $\mu (\{x \; | \; f(x)\neq^* g(x)\}) = 1$.

\item
$\mfb(\leq^*_\Null)$ is the least size of a $\leq^*_\Null$-unbounded set $A \subseteq \bbb$ i.e. $A$ is such that for each $f:\baire \to \baire$ there is a $g \in A$ so that the set of $\{x \; | \; \exists^\infty n \, f(x)(n) < g(x)(n)\}$ is not measure zero.

\item
$\mfd(\leq^*_\Null)$ is the least size of a $\leq^*_\Null$-dominating set $A \subseteq \bbb$ i.e. $A$ is such that for every $f:\baire \to \baire$ there is a $g \in A$ so that $\mu (\{x \; | \; f(x)\leq^* g(x)\}) = 1$.

\item
$\mfb(\in^*_\Null)$ is the least size of a $\in^*_\Null$-unbounded set $A \subseteq \bbb$ i.e. $A$ is such that for each $f:\baire \to \mathcal S$ there is a $g \in A$ so that the set of $\{x \; | \; \exists^\infty n \, g(x)(n) \notin f(x)(n)\}$ is not measure zero.

\item
$\mfd(\in^*_\Null)$ is the least size of a $\in^*_\Null$-dominating set $A \subseteq (\mathcal S)^{\baire}$ i.e. $A$ is such that for every $f:\baire \to \baire$ there is a $g \in A$ so that $\mu (\{x \; | \; f(x)\in^* g(x)\}) = 1$.

\end{enumerate}

The first goal is to prove the following theorem, which shows that for each ideal the six associated cardinals fit together as in the case of the corresponding fragment of Cicho\'n's diagram on $\omega$, note not all cardinals in the $\omega$ case have analogues here.

\begin{theorem}[The Higher Dimensional Cicho\'n diagram]
For an ideal $\mathcal I \in \{\Null, \Me, \Kb\}$ and interpreting $\to$ as \say{is $\ZFC$-provably less than or equal to} the following all hold:

\begin{figure}[h]
\centering
  \begin{tikzpicture}[scale=1.5,xscale=2]
     \draw 
           (1,0) node (Bin*) {$\mfb(\in_\mathcal I^*)$}
           (1,1) node (Bleq*) {$\mfb(\leq_\mathcal I^*)$}
           (1,2) node (Bneq*) {$\mfb(\neq_\mathcal I^*)$}
           (2,0) node (Dneq*) {$\mfd(\neq_\mathcal I^*)$}
           (2,1) node (Dleq*) {$\mfd(\leq_\mathcal I^*)$}
           (2,2) node (Din*) {$\mfd(\in_\mathcal I^*)$}
           ;
     \draw[->,>=stealth]
            (Bin*) edge (Bleq*)
            (Bleq*) edge (Bneq*)
            (Bin*) edge (Dneq*)
            (Bleq*) edge (Dleq*)
            (Bneq*) edge (Din*)
            (Dneq*) edge (Dleq*)
            (Dleq*) edge (Din*)
;
            
  \end{tikzpicture}
\end{figure}

\label{mainthm1}
\end{theorem}

\begin{proof}
Most of these implications are easy, however two are more substantial. The easy cases are shown below in Figure 3. The arguments for these are exactly identical to those in the Cicho\'n case. For instance, if $A$ is $\leq^*_\mathcal I$-bounded, then of course it is not $\leq^*_\mathcal I$-dominating hence $\mfb(\leq^*_\mathcal I) \leq \mfd(\leq^*_\mathcal I)$. Similarly, if $A \subseteq \bbb$ is a set so that there is a function $h:\baire \to \mathcal S$ so that for all $f \in A$ $f \in^*_\mathcal I h$ then $\hat{h} (x) (n) = {\rm max} \, h (x) (n) + 1$ witnesses the $\leq^*_\mathcal I$-bound of $A$. The other easy cases follow the same lines.

\begin{figure}[h]
\centering
  \begin{tikzpicture}[scale=1.5,xscale=2]
     \draw 
           (1,0) node (Bin*) {$\mfb(\in_\mathcal I^*)$}
           (1,1) node (Bleq*) {$\mfb(\leq_\mathcal I^*)$}
           (1,2) node (Bneq*) {$\mfb(\neq_\mathcal I^*)$}
           (2,0) node (Dneq*) {$\mfd(\neq_\mathcal I^*)$}
           (2,1) node (Dleq*) {$\mfd(\leq_\mathcal I^*)$}
           (2,2) node (Din*) {$\mfd(\in_\mathcal I^*)$}
           ;
     \draw[->,>=stealth]
            (Bin*) edge (Bleq*)
            (Bleq*) edge (Bneq*)
            (Bleq*) edge (Dleq*)
            (Dneq*) edge (Dleq*)
            (Dleq*) edge (Din*)
;
            
  \end{tikzpicture}
\caption{The easy cases}
\end{figure}

The two more substantial inequalities are $\mfb(\in^*_\mathcal I) \leq \mfd(\neq^*_\mathcal I)$ and $\mfb(\neq^*_\mathcal I) \leq \mfd(\in^*_\mathcal I)$, so I turn my attention to these. For the rest of this section, fix an ideal $\mathcal I \in\{ \Null, \Me, \Kb\}$.

The proofs of the inequalities consist of \say{lifting} the proofs for the Cicho\'n diagram to the higher dimensional case, particularly those in \cite{Bar1987}. Fix finite, disjoint subsets of $\omega$ which collectively cover $\omega$, say $\mathcal J = \{J_{n, k} \; | \; k < n\}$. Let $J_n = \bigcup_{k < n} J_{n, k}$. Let's say that a $\mathcal J$-function is a function $x: \omega \to \omega^{< \omega}$ so that for every $n$ we have that $x(n)$ has domain $J_n$. Similarly a $\mathcal J$-slalom is a function $s:\omega \to [\omega^{< \omega}]^{< \omega}$ so that for each $n$ $|s(n)| \leq n$ and if $w \in s(n)$ then the domain of $w$ is $J_n$. If $x$ is a $\mathcal J$ function and $s$ a $\mathcal J$-slalom then we let $x \in^* s$ if and only if for all but finitely many $n$ $x(n) \in s(n)$. Clearly via some simple coding we can find homeomorphisms/measure isomorphisms between $\omega^\omega$ and the set of $\mathcal J$-functions (with the obvious topology) and $\mathcal S$ and the set of $\mathcal J$-slaloms. It's then routine to verify that $\mathfrak{b}(\in^*_\mathcal I)$ is the same for $\in^*$ defined on slaloms and elements of Baire space or their $\mathcal J$-versions. 

\begin{lemma}
$\mathfrak{b} (\in^*_\mathcal I) \leq \mathfrak{d}(\neq^*_\mathcal I)$
\end{lemma}

\begin{proof}
I use the version of $\mathfrak{b}(\in^*_\mathcal I)$ defined in terms of $\mathcal J$-slaloms as in the paragraph before the statement of the lemma. Let $\kappa < \mathfrak{b} (\in_\mathcal I^*)$. I need to show that $\kappa < \mathfrak{d}(\neq^*_\mathcal I)$. Fix a set $A \subseteq (\omega^\omega)^{\omega^\omega}$ of size $\kappa$. Let's see that $A$ is not $\neq^*_\mathcal I$-dominating. To be clear, a set is $\neq^*_\mathcal I$ dominating if for every function $f:\omega^\omega \to \omega^\omega$ there is a $g \in A$ so that for all $x$ save for a set in $\mathcal I$ $f(x) \neq^* g(x)$. Negating this, we need to find a function $f:\omega^\omega \to \omega^\omega$ so that for all $g \in A$ the set $\{x \; | \; \exists^\infty n \, g(x)(n) = f(x) (n) \}$ is $\mathcal I$-positive. In fact, I will show that under the assumption, such an $f$ can be found so that each such set is $\mathcal I$-measure one.

Given an element of Baire Space, $x:\omega \to \omega$ let $x'$ be the $\mathcal J$-function defined by $x' (n) = x \hook J_n$. Note that since the $J_n$'s cover $\omega$ and are disjoint the function $x \mapsto x'$ is a bijection. Given a function $f:\omega^\omega \to \omega^\omega$ let $f '$ similarly be defined by letting $f'(x) = f(x) '$. Let $A' = \{g' \; | \; g \in A\}$. Since this set has size $\kappa$ it is $\in^*_\mathcal I$-bounded i.e. there is a function $f_A$ with domain the set of $\mathcal J$-functions and range the set of $\mathcal J$-slaloms so that for all $g' \in A'$ $\{x \; | \; g'(x) \notin^* f_A(x)\} \in \mathcal I$. I need to transform $f_A$ into a function $f$ as advertized in the previous paragraph. The crux of the argument is the following claim, which will also be used in Lemma \ref{bd2} below as well.

\begin{claim}
Given a $\mathcal J$-slalom $s$, there is a function $x_s:\omega \to \omega$ so that for all $y:\omega \to \omega$ if $y' \in^* s$ then there are infinitely many $n < \omega$ so that $x_s(n) = y (n)$.
\end{claim}

\begin{proof}[Proof of Claim]
Fix a $\mathcal J$-slalom $s$. For each $n$ let $s(n) = \{w^n_0, ..., w^n_{n-1}\}$. Define $x_s:\omega \to \omega$ by letting for each $n$ and $k < n$ and $l \in J_{n, k}$ $x_s (l) = w^n_k(l)$. Suppose now that $y: \omega \to \omega$ is such that $y' (n) \in s (n)$ for all but finitely many $n < \omega$. Fix some $n$ so that $y' (n) \in s(n)$, say, $y' (n) = w^n_k$. Then for each $l \in J_{n, k}$ $y (l) = w^n_k (l) = x_s(l)$. Since there are cofinitely many such $n$'s there are infinitely many such $l$'s so $x_s$ is as needed.
\end{proof}

Now, returning to the proof of the lemma, let $f:\omega^\omega \to \omega^\omega$ be defined by letting $f(x)$ be the function $x_{f_A(x)}$ in the terminology of the claim. In particular, if $g:\omega^\omega \to \omega^\omega$ then for every $x \in \omega^\omega$ if $g'(x) \in^* f_A(x)$ then there are infinitely many $n$ so that $g(x) (n) = f(x)(n)$. In particular the set $\{x \; | \; g'(x) \in^* f_A (x)\}$ is contained in the set $\{x \; | \; \exists^\infty n \, g(x)(n) = f(x)(n)\}$. For $g \in A$ the former is $\mathcal I$-measure one and so the latter is as well. As a result $f$ is as needed.
\end{proof}

By essentially dualizing the proof above we get as well the following.
\begin{lemma}
$\mathfrak{b}(\neq^*_\mathcal I) \leq \mathfrak{d}(\in^*_\mathcal I)$
\label{bd2}
\end{lemma}

\begin{proof}
Suppose $\kappa < \mfb (\neq^*_\mathcal I)$ and let $A \subseteq (\mathcal S)^{\baire}$ be of size $\kappa$. I need to show that there is an $f \in \bbb$ so that for each $h \in A$ the set of $x$ so that $f(x) \notin^* h(x)$ does not have $\mathcal I$-measure one. For each $h \in A$ let $g_h \in \bbb$ be defined by letting, for each $x \in \baire$ $g_h (x) = x_{h(x)}$ as defined in the claim of the previous lemma. In particular, for each $x$ note that if $f(x) \in^* h(x)$ then $\exists^\infty n \, g_h (x) (n) = f(x)(n)$. Now let $\bar{A} = \{g_h \; | \; h \in A\}$. This set has size at most $\kappa$ so there is a $\neq^*_\mathcal I$-bound by assumption, say $f$. This means that for each $g_h \in \bar{A}$ we have that $\{x \; | \; g_h (x) \neq^* f(x)\}$ is $\mathcal I$-measure one. But now the lemma is proved since for every $x$ so that $g_h (x) \neq^* f(x)$ by the contrapositive of the implication defining $g_h$ we have that $f(x) \notin^* h(x)$.
\end{proof}

Combining the easy cases shown in Figure 3 with the proofs of the above two lemmas then completes the proof of Theorem \ref{mainthm1}. 
\end{proof}

Using the fact that every set in $\Kb$ is meager, we get the following relation between the diagrams for $\Me$ and $\Kb$.

\begin{proposition}
The following inequalities are provable in $\ZFC$:

\begin{figure}[h]
\centering
  \begin{tikzpicture}[scale=1.5,xscale=2]
     \draw 
           (1,0) node (Bin*) {$\mfb(\in_\Me^*)$}
           (1,1) node (Bleq*) {$\mfb(\leq_\Me^*)$}
           (1,2) node (Bneq*) {$\mfb(\neq_\Me^*)$}
           (2,0) node (Dneq*) {$\mfd(\neq_\Me^*)$}
           (2,1) node (Dleq*) {$\mfd(\leq_\Me^*)$}
           (2,2) node (Din*) {$\mfd(\in_\Me^*)$}
	     (0,0) node (BinK*) {$\mfb(\in_\Kb^*)$}
           (0,1) node (BleqK*) {$\mfb(\leq_\Kb^*)$}
           (0,2) node (BneqK*) {$\mfb(\neq_\Kb^*)$}
           (3,0) node (DneqK*) {$\mfd(\neq_\Kb^*)$}
           (3,1) node (DleqK*) {$\mfd(\leq_\Kb^*)$}
           (3,2) node (DinK*) {$\mfd(\in_\Kb^*)$}
           ;
     \draw[->,>=stealth]
            (Bin*) edge (Bleq*)
            (Bleq*) edge (Bneq*)
            (Bin*) edge (Dneq*)
            (Bleq*) edge (Dleq*)
            (Bneq*) edge (Din*)
            (Dneq*) edge (Dleq*)
            (Dleq*) edge (Din*)
            (Dleq*) edge (DleqK*)
            (Dneq*) edge (DneqK*)
            (Din*) edge (DinK*)
            (BleqK*) edge (Bleq*)
            (BneqK*) edge (Bneq*)
            (BinK*) edge (Bin*)
            (BinK*) edge (BleqK*)
            (BleqK*) edge (BneqK*)
            (DneqK*) edge (DleqK*)
	      (DleqK*) edge (DinK*)

;
 \end{tikzpicture}
\end{figure}

\end{proposition}

\begin{proof}
Fix a relation $R$. To see that $\mfb (R_\Kb) \leq \mfb ( R_\Me)$ note that if $A \subseteq \bbb$ is $R_\Kb$-bounded, then it means that there is a function $f:\baire \to \baire$ so that for each $g \in A$ the set $\{ x \; | \; \neg g (x) R f(x)\}$ is $\leq^*$-bounded by some $z \in \baire$. But this means in particular that it is meager and hence for each $g \in A$ $g R_\Me f$. 

To see that $\mfd (R_\Me) \leq \mfd (R_\Kb)$, suppose that $A \subseteq \bbb$ is not $R_\Me$-dominating. This means that there is some $f:\baire \to \baire$ so that for each $g \in A$ the set $\{x \; | \; f(x) R g(x)\}$ is not comeager. It follows that in particular it is not $\Kb$-measure one then (since each such set is comeager) and therefore no $g \in A$ is a $R_\Kb$ bound on $f$, so $A$ is not $R_\Kb$-dominating.
\end{proof}

\section{ Relations between the Higher Dimensional Cardinals and Standard Cardinal Characteristics}

This section concerns the relationship between provable inequalities between the cardinals introduced previously and cardinal characteristics of the continuum. I look first at the relationship between the higher dimensional cardinals and cardinal $\mfc^+$ and then I compare the diagrams for the null and meager ideals. 

First let's note that for the $\mfb(R_\mathcal I)$ cardinals there are easy $\ZFC$ lower and upper bounds in terms of cardinal characteristics.
\begin{proposition}
For each $\mathcal I \in \{\Null, \Me, \Kb\}$ and $R \in \{ \in^*, \leq^*, \neq^*\}$ we have $\mfb(R) \leq \mfb(R_\mathcal I) \leq \mfb(R)^{{\rm non}(\mathcal I)}$.
\end{proposition}

\begin{proof}
Fix $R$ and $\mathcal I$. I start with the lower bound. Suppose $\kappa < \mfb(R)$ and let $A = \{f_\alpha \; | \; \alpha < \kappa\} \subseteq \bbb$. Define $g:\baire \to \baire$ (or $g:\baire \to \mathcal S$ in the case $R = \in^*$) by letting $g(x)$ be an $R$-bound on the set $\{f_\alpha (x) \; | \; \alpha < \kappa\}$. Such a bound exists by assumption that $\kappa < \mfb(R)$. But in this case $A$  is $R_\mathcal I$-bounded by $g$ since for each $\alpha < \kappa$ we have $f_\alpha (x) R g(x)$ for {\em every} $x \in \baire$, not just an $\mathcal I$-measure one set.

For the upper bound, fix an $R$-unbounded family of minimal size $B = \{x_\alpha \; | \; \alpha < \mfb(R)\}$. Let $C = \{y_\alpha \; | \; \alpha < {\rm non}(\mathcal I)\}$ be an $\mathcal I$-positive set of minimal size. Let $B^C$ be the set of functions from $C$ to $B$. Note that this set has size $\mfb(R)^{{\rm non}(\mathcal I)}$. Extend each element arbitrarily to the whole space. We claim that in fact this set is unbounded. Indeed, let $f:\baire \to \baire$ (or $f:\baire \to \mathcal S$ in the case $R = \in^*$). We need to show that $f$ is not a bound. To see this, let $h:C \to B$ be defined by letting $h(x)$ be not $R$-bounded by $f(x)$. Note that since $B$ is an unbounded family this is well defined. Extend $h$ arbitrarily to some $\bar{h}$. Now it can't be the case that $\bar{h} R_\mathcal I f$ since $\{x \; | \; \neg \bar{h}(x) R f(x)\}$ contains $C$, which is not in $\mathcal I$.
\end{proof}

I would like to argue that the standard diagonal arguments show that the cardinals defined above are greater than or equal to $\mathfrak{c}^+$, however this is not the case in $\ZFC$ alone. What I can show instead is that this holds under additional assumptions on certain cardinal characteristics on $\omega$. For the statement of the lemma below, recall that $non (\Kb) =  \mathfrak{b}$, see Theorem 2.8 of \cite{BlassHB}.
\begin{lemma}
For each $\mathcal I \in \{\Null, \Me, \Kb\}$ and $R \in \{ \in^*, \leq^*, \neq^*\}$, if $\mathfrak{b}(R) = non (\mathcal I) = \mfc$ then $\mfc^+ \leq \mathfrak{b}(R_\mathcal I)$. In particular, if $add (\Null) = \mfc$ then all 18 cardinals introduced in the previous section are greater than $\mfc$.
\end{lemma}

\begin{proof}
This is essentially a generalization of the standard diagonal arguments used to show that various cardinal characteristics are uncountable. The point is that in that case, the relations (on $\omega$) under consideration are always so that every finite set has an upper bound and the ideal is always the ideal of finite sets. It is exactly because arithmetic of cardinal characteristics is not so simple that the additional hypotheses are needed.

Fix $R$ and $\mathcal I$ and assume $\mathfrak{b}(R) = non (\mathcal I) = \mfc$. Let $f_\alpha : \baire \to \baire$ for each $\alpha < \mfc$. We want to find a $g:\baire \to \baire$ (or $g:\baire \to \mathcal S$ in the case of $R = \in^*$) so that for all $\alpha$ $f_\alpha R_\mathcal I g$. This is done as follows. First, list the elements of $\baire$ as $\{x_\alpha \; | \; \alpha < \mfc\}$. Next, note that for each $\beta < \mfc$, by the fact that $non (\mathcal I) = \mfc$ we have that $\{x_\alpha \; | \; \alpha < \beta\} \in \mathcal I$ and, by the fact that $\mathfrak{b}(R) = \mfc$ we have that for each $x_\gamma \in \baire$ the set $\{f_\alpha (x_\gamma) \; | \; \alpha < \beta\}$ has an $R$-bound, say $y^\gamma_\beta$. Now define $g$ so that $g(x_\alpha) = y^\alpha_\alpha$. It follows that for all $\alpha$ if $\gamma > \alpha$ then $f_\alpha (x_\gamma) R g(x_\gamma)$ and since the set $\{x_\gamma \; | \; \gamma > \alpha \}$ is $\mathcal I$-measure one we're done.
\end{proof}

Since writing and submitting this paper, in ongoing joint work with J. Brendle we have since shown that the cardinals of the form $\mfd(R_\mathcal I)$ are provably at least $\mfc^+$ but the cardinals of the form $\mfb(R_\mathcal I)$ can be both consistently equal to and strictly less than $\mfc$, even $\aleph_1$ with the continuum arbitrarily large. See \cite[Main Theorem 1.1]{BS21}\footnote{The anonymous referee independently observed an instance of this, and I thank them for the comment.}.

Finally in this section let me compare the cardinals for $\Me$ and $\Null$. Every argument given so far has worked equally well for each of them, and the theorem below suggests that this is not an accident.

\begin{theorem}
If $add (\Null) = cof (\Null)$ then for every relation $R \in\{ \in^*, \leq^*, \neq^*\}$ we have that $\mfb (R_\Null) = \mfb (R_\Me)$ and $\mfd (R_\Null) = \mfd (R_\Me)$.
\label{M=N}
\end{theorem}

The proof of this theorem follows immediately from the following two lemmas, the first of which is well known.
\begin{lemma}[Theorem 2.1.8 of \cite{BarJu95}]
If $add(\Null) = cof(\Null)$ then there is a bijection $f:\baire \to \baire$ so that for all $A \subseteq \baire$ $f(A) \in \Null$ if and only if $A \in \Me$ and $f(A) \in \Me$ if and only if $A \in \Null$.
\label{erdosthm}
\end{lemma}

\begin{lemma}
If there is a bijection $f:\baire \to \baire$ as in Lemma \ref{erdosthm} then for every relation $R \in \{ \in^*, \leq^*, \neq^*\}$ we have that $\mfb (R_\Null) = \mfb (R_\Me)$ and $\mfd (R_\Null) = \mfd (R_\Me)$.
\end{lemma}

\begin{proof}
Fix a relation $R \in \{ \in^*, \leq^*, \neq^*\}$ and let $f:\baire \to \baire$ be a bijection as described in Lemma \ref{erdosthm}. First, suppose that $\kappa < \mfb (R_\Null)$ and let $A \subseteq ( \baire )^{\baire}$ be a set of size $\kappa$. I claim that there is a function $g_A:\baire \to \baire$ ($g_A:\baire \to \mathcal S$ in the case $R = \in^*$) so that for all $g \in A$ $g R_\Me g_A$ and hence $\kappa < \mfb (R_\Me)$. Let $A_f = \{ g \circ f \; | \; g \in A\}$. Since $f$ is a bijection $|A_f| = \kappa$. By the hypothesis, let $\bar{g}:\baire \to \baire$ be an $R_\Null$ bound on $A_f$. I claim that $g_A = \bar{g} \circ f^{-1}$ is as needed. We have that for every $g \in A$ if $f^{-1}(x) = y$ is in the measure one set for which $g(f(y)) R \bar{g} (y)$ is true then the following holds:

\begin{equation*}
g(x) =  g (f(f^{-1}(x)) R \bar{g} ( f^{-1} (x)) = g_A (x)
\end{equation*}
\linebreak
Therefore $f^{-1} (\{x \; | \; \neg g(x) R g_A (x)\})$ is contained in $\{x \; | \; \neg g(f(x)) R \bar{g}(x)\}$, which is null by assumption and so the former is null as well. Hence by the property of $f$ it follows that $\{x \; | \; \neg g(x) R g_A(x)\}$ is meager so $g_A$ is an $R_\Me$-bound as needed.

This shows that $\mfb (R_\Null) \leq \mfb (R_\Me)$ however an identical argument, flipping the roles of the meager and null sets, shows the reverse inequality so we get that $\mfb (R_\Null) = \mfb (R_\Me)$. 

An essentially dual argument works to show that $\mfd (R_\Null) = \mfd (R_\Me)$. Let me sketch it, though I leave out the details. Assuming that $\kappa < \mfd (R_\Null)$ we fix a set $A \subseteq (\baire)^{\baire}$ of size $\kappa$, define $A_f$ as before and let $\bar{g}$ be a function not dominated by any member of $A_f$. Then essentially the same argument shows that $\bar{g} \circ f^{-1}$ is a function not dominated by any member of $A$ and, again by symmetry we obtain the required equality.
\end{proof}

J. Brendle and I in the above mentioned ongoing work have shown that without additional assumptions both $\mfb(R_\Me) < \mfb(R_\Null)$ and $\mfb(R_\Null) < \mfb(R_\Me)$ are consistent for every $R \in \{\in^*, \leq^*, \neq^*\}$. However, $\mfd(R_\Null) = \mfd(R_\Me)$ in $\ZFC$ for every $R$, see \cite[Main Theorem 1.4]{BS21}. 

\section{Consistency Results}
In this section I consider consistent separations between the cardinals. For readability, I focus on the case of $\mathcal I = \Null$, however, it's routine to check that the arguments go through for $\mathcal I = \Me$ and $\mathcal I = \Kb$. Indeed the essential point will be simply that $\Null$ has a Borel base and contains the countable subsets of $\baire$. Also, I will only be considering models of $\CH$ so by Theorem \ref{M=N} any separation between nodes in the $\Null$ diagram will hold equally for the $\Me$ diagram.

From now on assume $\GCH$ holds and fix an enumeration of $\baire$ in order type $\omega_1$, say $\{x_\alpha \; | \; \alpha < \omega_1\}$. Also fix an enumeration of the Borel sets in $\Null$ in order type $\omega_1$, say $\{N_\alpha \; | \; \alpha < \omega_1\}$.  Suppose we have some forcing notion $\mathbb P$ which does not add reals. Note that in this case if $B$ is a Borel set then $\mathbb P$ forces that the name for $B$ is equal to its evaluation in the ground model. Also, since $\mathbb P$ does not add reals, it does not add any Borel sets either. This translates to the following idea, which is used in several proofs. Suppose $\dot{A}$ is a $\mathbb P$ name for a subset of $\baire$. If for some condition $p \in \mathbb P$ we have that $p \forces\mu( \dot{A} ) = \check{0}$, then we can always find a $q \leq p$ and a Borel null set in the ground model $N$ so that $q \forces \dot{A} \subseteq \check{N}$.

The following simple lemma will be used in several proofs.

\begin{lemma}
Let $\vec{N}_0 = \langle N_{0, \alpha} \; | \; \alpha < \omega_1\rangle$ and $\vec{N}_1 = \langle N_{1, \alpha} \; | \; \alpha < \omega_1\rangle$ be two enumerations of Borel null sets of $\baire$, possibly with repetitions, in order type $\omega_1$ so that for each $x \in \baire$ there are uncountably many $\alpha, \beta$ so that $x \notin N_{0, \alpha} \cup N_{1, \beta}$. Then there is an enumeration in order type $\omega_1$, say $\langle(N'_{0, \alpha}, N'_{1, \alpha} ) \; | \; \alpha <\omega_1\rangle$ of the set of all pairs $(N_{0, \beta}, N_{1, \gamma})$ so that for each $\alpha < \omega_1$ we have $x_\alpha \notin N'_{0, \alpha} \cup N'_{1, \alpha}$.
\label{mainlemma}
\end{lemma}

\begin{proof}
First fix any enumeration of $\vec{N}_0 \times \vec{N}_1$, say $\langle (N''_{0, \alpha}, N''_{1, \alpha}) \; | \; \alpha < \omega_1\rangle$ and define inductively for each $\alpha$ $(N'_{0, \alpha}, N'_{1, \alpha})$ to be the least $\gamma$ so that $(N''_{0, \gamma}, N''_{1, \gamma})$ has not yet been enumerated and $x_\alpha \notin  N''_{0, \gamma} \cup N''_{1, \gamma}$. Observe first that there is some such $\gamma$ since by assumption there are uncountably many pairs $(N_{0, \gamma_0}, N_{1, \gamma_1})$ whose union does not contain $x_\alpha$ and only countably many from each list has appeared so far. It remains to show that every $(N''_{0, \gamma}, N''_{1, \gamma})$ gets enumerated under this procedure. Suppose not and let $\gamma$ be least so that $(N''_{0, \gamma}, N''_{1, \gamma})$ is not enumerated. Since for every $\beta < \gamma$ the pair $(N''_{0, \beta}, N''_{1, \beta})$ was enumerated, there was some countable stage by which this happened and so for cocountably many $\alpha$ it must have been the case that $x_\alpha \in N''_{0, \gamma} \cup N''_{1, \gamma}$. But this is impossible since $N''_{0, \gamma} \cup N''_{1, \gamma}$ is measure zero and hence cannot contain a cocountable set.
\end{proof}

\subsection{Generalizing Cohen Forcing}

The point of this subsection is to prove the following theorem.

\begin{theorem}[$\GCH$]
Let $\kappa > \aleph_2$ be regular. There is a cofinality preserving forcing notion $\mathbb P_\kappa$ so that if $G \subseteq \mathbb P_\kappa$ is $V$-generic then in $V[G]$ we have $\mfc^+ = \aleph_2 = \mathfrak{b} (\neq^*_\Null) < \mfd (\neq^*_\Null) = 2^\mfc = \kappa$.
\label{mainthm2}
\end{theorem}

The proof will involve an iteration of length $\kappa$ of a certain forcing notion, $\mathbb{C}_\Null$. Let me begin by introducing this forcing notion and studying its properties.
\begin{definition}
The $\Null$-Cohen forcing, denoted $\mathbb{C}_\Null$, is the set of all $p:{\rm dom} (p) \subseteq \baire \to \baire$ so that ${\rm graph}(p)$ is Borel and ${\rm dom} (p)$ is measure zero. We let $p \leq q$ if and only if $p \supseteq q$.
\end{definition}

The following observations are easy but will be useful.
\begin{proposition}
The forcing $\mathbb{C}_\Null$ is $\sigma$-closed and has size $\mfc$, hence it has the $\mfc^+$-c.c. In particular, under $\CH$ all cofinalities and hence cardinalities are preserved.
\end{proposition}

\begin{proof}
First let's see that $\mathbb C_\Null$ is $\sigma$-closed. Given a descending sequence $p_0 \geq p_1 \geq p_2 ...$ let $p = \bigcup_{n < \omega} p_n$. Since the countable union of Borel sets is Borel it follows that $p$ has a Borel graph and since the countable union of null sets is null, it follows that $p$ has null domain. Thus $p$ is a condition so it is a lower bound on the sequence of $p_n$'s. 

To see that $\mathbb C_\Null$ has size $\mfc$ it suffices to note that each condition is a Borel subset of $(\omega^\omega)^2$, of which there are only $\mfc$ many.
\end{proof}

Note that since $\mathbb{C}_\Null$ adds no reals or Borel sets and every condition $p \in \mathbb{C}_\Null$ is a Borel set it follows that $\forces_{\mathbb{C}_\Null} \dot{\mathbb{C}}_\Null = \check{\mathbb{C}}_{\Null}$ and so in particular, the product and iteration of $\mathbb{C}_\Null$ are the same. Now, a straightforward density argument shows that $\mathbb{C}_\Null$ adds a function $g:\omega^\omega \to \omega^\omega$, namely the union of the generic filter. Indeed it's easy to see that if $p$ is any condition and $N$ is any Borel null set then there is a condition $q \leq p$ so that $N \subseteq {\rm dom}(q)$. I need to verify two properties of $\mathbb{C}_\Null$, given as Lemmas \ref{lemma1} and \ref{lemma2} below. The first will imply that in an iterated extension $\mfd(\neq^*_\Null)$ becomes large and the second will imply that $\mfb(\neq^*_\Null)$ remains small in an iterated extension. 

\begin{lemma}
If $G \subseteq \mathbb{C}_\Null$ is generic over $V$ then in $V[G]$ the set of $f \in (\baire)^{\baire} \cap V$ is not dominating with respect to the relation $\neq^*_\Null$.
\label{lemma1}
\end{lemma}

\begin{proof}
In fact a stronger statement is true, namely if $g = \bigcup G$ and $\dot{g}$ is the name for $g$, then for any $f:\omega^\omega \to \omega^\omega$ in the ground model the set $\{x \; | \; f(x) = g(x)\}$ is not measure zero. To see this, suppose that for some condition $p$ and ground model function $f$ we have that $p \forces \mu (\{x \; | \; \check{f}(x) = \dot{g}(x)\}) = \check{0}$. Since every null set is contained in a Borel null set, there is a Borel Null set $N$, necessarily in the ground model since $\mathbb{C}_\Null$ is $\sigma$-closed, and a strengthening $q \leq p$ so that $q \forces \{x \; | \; \check{f}(x) = \dot{g} (x)\} \subseteq \check{N}$. But now let $x \notin N \cup {\rm dom}(q)$ (this is possible since $N \cup {\rm dom}(q) \in \mathcal N$). It is straightforward to verify that $q^* = q \cup \{\langle x, f(x)\rangle \}$ is a condition extending $q$ but clearly $q^* \forces \{x \; | \; \check{f}(x) = \dot{g}(x)\} \nsubseteq \check{N}$, which is a contradiction. It follows in particular that for every $f \in V$ we have that on a non null set of $x$ there are infinitely many $n < \omega$ so that $f(x) (n) = g(x) (n)$. This implies the lemma.
\end{proof}

\begin{lemma}
If $G \subseteq \Pi_I \mathbb{C}_\Null$ is generic over $V$ for the countable support product of $\mathbb C_\Null$ over an index set $I$ of size at most $\aleph_1$ then in $V[G]$ the set of $f \in (\baire)^{\baire} \cap V$ is unbounded with respect to the relation $\neq^*_\Null$.
\label{lemma2}
\end{lemma}

\begin{proof}
I need to show that in $V[G]$ there is no $h:\baire \to \baire$ so that for all $f:\baire \to \baire$ in $V$ the set of $x$ for which $f(x) \neq^* h(x)$ is measure one. Thus suppose for a contradiction that there is a condition $p$ and a name $\dot{h}$ so that $p \forces \dot{h}: \check{\baire} \to \check{\baire}$ is such a function. I need to define in $V$ a function for which this fails. 

Note that (under $\CH$) $\Pi_I \mathbb{C}_\Null$ has size $\aleph_1$. For each condition $p \in  \Pi_I \mathbb{C}_\Null$, let $N^p$ be the union of the domains of the coodinate conditions. Since $ \Pi_I \mathbb{C}_\Null$ has countable support, it follows that $N^p$ is null. Now, using Lemma \ref{mainlemma} fix an enumeration $\langle (N_{0, \alpha}, p_\alpha) \; | \; \alpha < \omega_1 \rangle$ of all pairs where $N_{0, \alpha}$ ranges over the Borel null sets $N_\alpha$ and $p_\alpha$ is a condition in $\Pi_I \mathbb{C}_\Null$, and $x_\alpha \notin N_{0, \alpha} \cup N^{p_\alpha}$. For each $\alpha$, let $r_\alpha \leq p_\alpha$ decide $\dot{h}(x_\alpha)$. Say that $r_\alpha \forces \dot{h}(\check{x}_\alpha) = \check{y}_\alpha$ for some $y_\alpha$. Let $h^* :\baire \to \baire$ be the function (defined in $V$) so that $h^*(x_\alpha) = y_\alpha$ for all $\alpha$. Suppose that there is some Borel null set $N$ and some condition $p$ which forces that $\{x \; | \; \exists^\infty n \, \dot{h} (x) (n) = \check{h}^* (x) (n) \} \subseteq N$. Let $\alpha$ be such that $(N, p) = (N_{0, \alpha}, p_\alpha)$. Then $p_\alpha \forces \{x \; | \; \exists^\infty n \, \dot{h} (x) (n) = \check{h}^* (x) (n) \} \subseteq N_{0, \alpha}$. But $r_\alpha \leq p_\alpha$ forces that $\dot{h} (x_\alpha) = \check{h}^* (x_\alpha)$ and by the choice of enumeration we had that $x_\alpha \notin N_{0, \alpha}$, which is a contradiction.
\end{proof}

I'm now ready to prove Theorem \ref{mainthm2}. In fact it follows from the following theorem, which is just a more precise statement of what will be shown.
\begin{theorem}
Let $\kappa$ be a regular cardinal greater than $\aleph_2$ and let $\mathbb P_\kappa$ be the countable support product of $\mathbb{C}_\Null$. Then $\mathbb P_\kappa$ preserves cofinalities and cardinals and if $G\subseteq \mathbb P_\kappa$ is $V$-generic then in $V[G]$ $\mfc^+ = \aleph_2 = \mathfrak{b} (\neq^*_\Null) < \mfd (\neq^*_\Null) = 2^\mfc = \kappa$.
\end{theorem}

\begin{proof}
Fix $\kappa > \aleph_2$ regular, let $\mathbb P = \mathbb P_\kappa$ be the countable support product of $\mathbb{C}_\Null$ of length $\kappa$. Clearly $\mathbb P$ is $\sigma$-closed and a straightforward $\Delta$-system argument using $\GCH$ shows that it has the $\aleph_2$-c.c. It follows that all cardinals and cofinalities are preserved.

Also, for each $\alpha$ the $\alpha$-stage forcing $\mathbb P_\alpha$ adds a new function $g_\alpha:\baire \to \baire$ so in the extension $2^\mfc \geq \kappa$. A standard nice name counting argument, again using $\GCH$ shows that in fact $2^\mfc = \kappa$. 

It remains to show that $\aleph_2 = \mathfrak{b} (\neq^*_\Null)$ and $\mfd (\neq^*_\Null) = \kappa$. For the first of these, it suffices to see that $(\baire)^{\baire} \cap V$ is unbounded with respect to $\neq^*_\Null$. To see this, by Lemma \ref{lemma2}, it suffices to note that if $\dot{f}$ is a name for a function in $\bbb$ then $\dot{f}$ is equivalent to a $\Pi_I \mathbb{C}_\Null$ for $I$ an index set of size $\aleph_1$. This latter statement is proved as follows: let, for each $x \in \baire$ $\mathcal A_x$ be an antichain of conditions deciding $\dot{f}(\check{x})$ and note that the cardinality of the supports of the elements of $\bigcup_{x \in \baire} \mathcal A_x$ has size $\aleph_1$ by $\CH$ using the countable support of the product.

Finally for $\mfd (\neq^*_\Null) = \kappa$, suppose that $A \subseteq (\baire)^{\baire}$ of size $< \kappa$. It follows that $A$ must have been added by some initial stage of the iteration, the next stage of which killed the possibility that it was dominating by Lemma \ref{lemma1}.
\end{proof}

Let me reiterate that, defining $\mathbb{C}_\Me$ and $\mathbb{C}_\Kb$ in the obvious way the proofs can be repeated verbatim to obtain similar consistencies for the $\Me$ and $\Kb$ ideals. The same is true in the remaining subsections, though I won't explicitly say this again. An interesting open question though is the following.
\begin{question}
Are the forcing notions $\mathbb C_\Null$, $\mathbb C_\Me$ and $\mathbb C_\Kb$ forcing equivalent?
\end{question}

\subsection{Generalizing Hechler Forcing}

In this subsection I consider a generalization of Hechler forcing called $\mathbb D_\Null$ and look at two models obtained by iterating this forcing. First I consider the countable support iteration of $\mathbb D_\Null$ and then I look at a non-linear iteration of $\mathbb D_\Null$. The use of non-linear iterations of Hechler forcing (for $\baire$) originates in the work of Hechler in \cite{Hechler74} and was generalized by Cummings and Shelah in \cite{CuSh95} to accomodate $\kappa^\kappa$. Here I follow the same expositional flow as used in \cite{CuSh95}. Using this forcing I obtain the following consistency result.
\begin{theorem}
Let $\aleph_2< \kappa \leq \lambda$ with $\kappa$ and $\lambda$ regular. Then there is a forcing notion $\mathbb P_{\kappa, \lambda}$ which preserves cardinals and cofinalities such that if $G \subseteq \mathbb P_\kappa$ is generic then in $V[G]$ we have that $\mfb (\leq^*_\Null) = \kappa \leq \mfd (\leq^*) = 2^\mfc = \lambda$.
\label{mainthmhechler}
\end{theorem}

Similar to the last subsection I start by introducing the one step and studying its properties. As before, I work with the null ideal for definiteness but it's easy to see that the proofs adapt to the case of the other ideals.
\begin{definition}
The $\Null$-Hechler forcing $\mathbb{D}_\Null$ consists of the set of pairs $(p, \mathcal F)$ where $p \in \mathbb{C}_\Null$ and $\mathcal F$ is a countable set of functions $f:\baire \to \baire$. We let $(p, \mathcal F) \leq (q, \mathcal G)$ in case $p \supseteq q$, $\mathcal F \supseteq \mathcal G$ and for all $x \in {\rm dom}(p) \setminus {\rm dom}(q)$ and all $g \in \mathcal G$ $g(x) \leq^* p(x)$.
\end{definition}

If $d= (p_d, \mathcal F_d) \in \mathbb{D}_\Null$, let me call $p_d$ the {\em stem} of the condition and $\mathcal F_d$ the {\em side part}. The basic properties I will need for $\mathbb{D}_\Null$ are as follows.
\begin{proposition}
$\mathbb{D}_\Null$ is $\sigma$-closed and has the $\mfc^+$-c.c., thus assuming $\CH$, it preserves cofinalities and cardinals. Also, if $G \subseteq \mathbb D_\Null$ is $V$-generic then the union of $G$ is a function $g:\baire \to \baire$ so that for any $f:\baire \to \baire$ in the ground model the set of $x$ so that $g(x)$ does not eventually dominate $f(x)$ is null.
\end{proposition}

\begin{proof}
That $\mathbb{D}_\Null$ is $\sigma$-closed is the same as the proof for $\mathbb{C}_\Null$. To see that it has the $\mfc^+$-c.c. it suffices to note that if two conditions have the same stem then they are compatible.

Now to see that $g$ is total is a simple density argument, noting that if $d$ is some condition and $x \notin {\rm dom}(p_d)$ then there is a $y$ dominating all of the $f(x)$ for $f \in \mathcal F_d$ since this set is countable and hence $(p_d \cup \{\langle x, y\rangle \}, \mathcal F_d)$ extends $d$ as needed. Moreover, if $f:\baire \to \baire$ is a function in the ground model and $d$ is any condition then clearly we can strengthen $d$, say to $d'$ so that $f$ is included in the side part $d'$. This strengthening forces, by the definition of the extension relation, that for all $x \notin {\rm dom}(p_{d'})$ $f(x) \leq^* \dot{g} (x)$. Since the domain of $p_{d'}$ was measure zero this proves the second part.
\end{proof}

\begin{remark}
While it's not used in any proof let me note that, unlike with $\mathbb C_\Null$ it is {\em not} the case that every condition in $\mathbb D_\Null$ can be extended to include any Borel null set in the domain of its stem. This is because given any uncountable Borel set, under $\CH$ one can use a simple diagonal argument to build a function which is not dominated by any Borel function on that set. What is true however, is that the stem of any condition can be extended to include any {\em countable} set.
\end{remark}

Let me now show what happens in the generic extension by a countable support iteration of $\mathbb{D}_\Null$.
\begin{theorem}
Let $\kappa$ be regular and let $\mathbb P_\kappa$ be the countable support iteration of $D_\Null$. If $G \subseteq \mathbb P_\kappa$ is $V$-generic then in $V[G]$ $\mfb (\leq^*_\Null) = \mfd (\neq_\Null^*) = \kappa = 2^\mfc$.
\label{mainthmhechler2}
\end{theorem}

\begin{proof}
Let $\kappa > \aleph_2$ be regular and let $\mathbb P_\kappa$ be the countable support iteration of length $\kappa$ of $\mathbb{D}_\kappa$. That cardinals and cofinalities are preserved follows as for $\mathbb{C}_\Null$. Also, every set $A \subseteq \bbb$ of size less than $\kappa$ is added by some initial stage, after which a function bounding $A$ was added so $\mfb (\leq^*_\Null) = \kappa$. Moreover a nice name argument easily gives that $2^\mfc = \kappa$.

It remains to see that $\mfd (\neq_\Null^*) = \kappa$ in this model. For this, I use the fact that countable support iterations always add a generic for $Add (\omega_1, 1)$ at limit stages of cofinality $\omega_1$. Now, given any function $f:\omega_1 \to \omega_1$ we can think of it as a function $\hat{f}$ from $\baire$ to $\baire$ by letting $\hat{f}(x_\alpha) = x_\beta$ just in case $f(\alpha) = \beta$. Suppose $A \subseteq \bbb$ is a set of size less than $\kappa$. It must have been added by some initial stage of the iteration $\mathbb P_\kappa$ and therefore there is a later stage which adds an $Add (\omega_1, 1)$-generic function $g:\omega_1 \to \omega_1$. By density, given any $f:\baire \to \baire$ in $A$ and any Borel null set $N$ we can find an $x \notin N$ so that $\hat{g}(x) = f(x)$ and therefore, for any $f \in A$ the set of $x$ for which $\hat{g}(x) = f(x)$ is not null. Therefore in particular $A$ is not a $\neq^*_\Null$-dominating family. Thus $\mfd (\neq^*_\Null) = \kappa$.
\end{proof} 

I'm now ready to prove Theorem \ref{mainthmhechler}. This uses a version of the iteration discussed in Section 3 of \cite{CuSh95}, it being based on the original non-linear iteration of Hechler in \cite{Hechler74}. Let me recall the basics of what I need. Fix $\kappa < \lambda$ regular cardinals greater than $\aleph_2$ and let $\mathbb{Q} = (Q, \leq_\mathbb Q)$ be a well founded partial order so that $\mfb(\mathbb Q) = \kappa$ and $\mfd (\mathbb Q) = \lambda$. For example, under $\GCH$, $\kappa \times [\lambda]^{< \kappa}$ ordered by $(\alpha, \tau) \leq (\beta, \sigma)$ if and only if $\alpha < \beta$ and $\tau \subseteq \sigma$ is such an order, see Lemma 2 of \cite{CuSh95} for a proof. I need to define a $\sigma$-closed, $\aleph_2$-c.c. forcing notion $\mathbb{D} (\mathbb Q)$ so that forcing with this partial order adds a cofinal embedding of $\mathbb{Q}$ into $(\bbb, \leq^*_\Null)$. If I can do this, then by Lemmas 3 and 5 of \cite{CuSh95} it follows that in the extension by this forcing notion $\mfb (\leq^*_\Null) = \kappa$ and $\mfd (\leq^*_\Null) = \lambda$. For completeness, here are the cited lemmas.

\begin{lemma}[Lemma 3 of \cite{CuSh95}]
If $\mathbb P$ and $\mathbb Q$ are partially ordered sets and $\mathbb P$ embeds cofinally into $\mathbb Q$ then $\mfb(\mathbb P) = \mfb (\mathbb Q)$ and $\mfd (\mathbb P) = \mfd (\mathbb Q)$.
\end{lemma}

\begin{lemma}[Lemma 5 of \cite{CuSh95}]
Suppose $\mathbb P$ is a partial order with $\mfb (\mathbb P) = \beta$ and $\mfd (\mathbb P) = \delta$.
\begin{enumerate}
\item
If $V[G]$ is a generic extension of $V$ so that every set of ordinals of size less than $\beta$ in $V[G]$ is covered by a set of ordinals of size less than $\beta$ in $V$ then $V[G] \models \mfb (\mathbb P) = \beta$.
\item
If $V[G]$ is a generic extension of $V$ so that every set of ordinals of size less than $\delta$ in $V[G]$ is covered by a set of ordinals of size less than $\delta$ in $V$ then $V[G] \models \mfd (\mathbb P) = \delta$.
\end{enumerate}

\end{lemma}

Let me define the forcing notion I need. In what follows, if $a \in \mathbb Q$ let $\mathbb Q \hook a = \{b \in \mathbb Q \; | \; b < a\}$. Similarly if $p$ is a function with domain contained in $\mathbb Q$, let $p \hook  a$ be the restriction of $p$ to $\mathbb Q \hook a$.
\begin{definition}
Let $\mathbb{Q}_{top}$ be $\mathbb{Q}$ with the addition of a top element, $top$, greater than all other elements. For each $a \in \mathbb Q_{top}$ define inductively a forcing notion $\mathbb{D}(\mathbb Q)_a$ to be the set of functions $p$ with ${\rm dom}(p) \subseteq\mathbb Q \hook a$ countable and for each $b \in {\rm dom}(p)$ $p(b)$ is a $\mathbb{D} (\mathbb Q)_b$-name for an element of $\mathbb{D}_\Null$ of the form $(\check{p}, \dot{\mathcal F})$. Let $p \leq_{\mathbb{D} (\mathbb Q)_a} q$ if and only if $p \supseteq q$ and for every $b \in {\rm dom}(q)$ we have that $p \hook b \forces_{\mathbb D (\mathbb Q)_b} p(b) \leq_{\mathbb{D}_\Null} q(b)$. Finally we let $\mathbb{D} (\mathbb Q) = \mathbb D(\mathbb Q)_{top}$.
\end{definition}

\begin{remark}
Below I show that $\mathbb D (\mathbb Q)_a$ is $\sigma$-closed for each $a \in \mathbb Q$. It follows that there is no loss in generality in insisting that for all $b$ the name for the stem of $p(b)$ is a check name, since the latter is always coded by a real and hence the set of conditions like this is dense.
\end{remark}

\begin{lemma}
For every $a \in \mathbb Q$, the partial order $\mathbb{D} (\mathbb Q)_a$ is $\sigma$-closed, and under $\GCH$, has the $\aleph_2$-c.c.. 
\end{lemma}

\begin{proof}
Fix $a \in \mathbb Q$. Suppose that $p_0 \geq p_1 \geq ...\geq p_n \geq ...$ is a decreasing sequence of elements in $\mathbb D(\mathbb Q)_a$. Let $p$ be defined as the function whose domain is the union of the domains of all of the $p_n$'s and so that for each $b \in {\rm dom} (p)$ we let $p(b)$ name a lower bound on the set of $\{p_n (b) \;  |\; n < \omega \, {\rm and} \, b \in {\rm dom} (p_n)\}$. Since $\mathbb D_\Null$ is $\sigma$-closed such a name exists. Clearly $p$ is a lower bound on the sequence so $\mathbb D (\mathbb Q)_a$ is $\sigma$-closed.

To see that $\mathbb D (\mathbb Q)_a$ has the $\aleph_2$-c.c., suppose that $A = \{p_\alpha \; | \; \alpha < \omega_2\}$ is a set of conditions. Applying the $\Delta$-system lemma (by $\GCH$) we can thin out $A$ to a $\Delta$-system so that any two conditions' domains coincide on some countable set $B \subseteq \mathbb Q \hook a$. But now, since each name for the stem in $p(b)$ is a check name, the set of all possible sequences of stems on the coordinates in $B$ has size $\omega_1^\omega = \omega_1$ (using $\CH$). Thus $A$ contains $\omega_2$ many conditions so that on $B$ the stems agree, and each such condition's overlapping domains is $B$. But this means that those conditions are all compatible.
\end{proof}

The next lemma is entirely straightforward to verify.
\begin{lemma}
Suppose $a < b \in \mathbb Q_{top}$ then $\mathbb D (\mathbb Q)_a$ completely embeds into $\mathbb D(\mathbb Q)_b$ and the map $\pi:\mathbb D (\mathbb Q)_b \to \mathbb D(\mathbb Q)_a$ defined by $\pi (p) = p\hook a$ is a projection.
\label{projection}
\end{lemma}

Now we get to the heart of the matter. Let $G \subseteq\mathbb{D} (\mathbb Q)$ be generic over $V$. For each $a \in \mathbb Q$ let $f_G^a : \baire \to \baire$ be the $\mathbb{D}_\Null$-generic function added by the $a^{\rm th}$ coordinate i.e. $f_G^a = \bigcup_{p \in G} \{ p (a)_0\}$

\begin{lemma}
The map $a\mapsto f_G^a$ is a cofinal mapping of $\mathbb{Q}$ into $(\bbb, \leq^*_\Null)$.
\label{cof}
\end{lemma}

\begin{proof}
I need to show that for $a, b \in \mathbb Q$, first of all that $a < b$ if and only if $f_G^a \leq^*_\Null f_G^b$ and second of all that for each $f \in \bbb$ there is an $a \in \mathbb Q$ so that $f \leq^*_\Null f_G^a$. First suppose that $a < b$. By Lemma \ref{projection} for any $p \in \mathbb D(\mathbb Q)$ we can find a  strengthening $q$ so that $q(b)$ forces that $f_G^a \leq^*_\Null f_G^b$ since $f_G^a$ is added at an earlier stage.

Now suppose that $a$ and $b$ are incomparable (the case where $b < a$ is symmetric to the above). Suppose for a contradiction that these is some condition $p \in \mathbb \mathbb D (\mathbb Q)$ so that $p \forces \{x \; | \; \dot{f}_G^a(\check{x}) \nleq^* \dot{f}_G^b (\check{x}) \} \subseteq \check{N}$, so in particular $p$ forces that $f_G^a \leq^*_\Null f_G^b$. By strengthening if necessary we can assume that $a, b \in {\rm dom}(p)$. Now choose $\alpha$ so that $x_\alpha$ is not in $N$, ${\rm dom} (p (a)_0)$ or ${\rm dom}(p(b)_0)$. Since all three are null sets, their union is null so there is such an $x_\alpha$. Now let $q_b \leq_{\mathbb D(\mathbb Q)_b} p \hook b$ be a strengthening so that if $\dot{\mathcal F}_b$ is the name of the side part of $p(b)$ then $q_b$ decides the check name values of all countably many elements of $\{\dot{f} (\check{x}_\alpha) \; | \; \dot{f} \in \dot{\mathcal F}_b\}$, this is possible by the fact that the forcing is $\sigma$-closed. Let $p_b$ be the condition obtained by letting $p_b(x) = p(x)$ if $x \notin {\rm dom}(q)$ and $p_b(x) = q(x)$ otherwise. Now let $q_a \leq p_b \hook a$ strengthen $p_b$ to decide the values $\{\dot{f} (\check{x}_\alpha) \; | \; \dot{f} \in \dot{\mathcal F}_a\}$ for $\dot{\mathcal F}_a$ the name of the side part of $p(a)$. Finally let $q$ be the condition which agrees with $q_a$ on its domain and agrees with $p_b$ otherwise. Note that since $a$ and $b$ are incomparable in $\mathbb Q$ neither $a$ nor $b$ is in the domains of $q_a$ or $q_b$ and so $q (a) = p(a)$ and $q(b) = p(b)$ and $q \leq p$. Now, let $x_b$ be a $\leq^*$-bound on the set $\{\dot{f} (\check{x}_\alpha) \; | \; \dot{f} \in \dot{\mathcal F}_b\}$ and let $x_a$ be such that $x_b + 1 \leq^* x_a$ and $x_a$ is a $\leq^*$-bound on the set $\{\dot{f} (\check{x}_\alpha) \; | \; \dot{f} \in \dot{\mathcal F}_a\}$. Finally let $q^*$ be the strengthening of $q$ so that the stem of $q^*(a)$ includes $(x_\alpha, x_a)$ and the stem of $q^*(b)$ includes $(x_\alpha, x_b)$. Then $q^* \forces \dot{f}_G^b (\check{x}_\alpha) + 1 \leq^* \dot{f}_G^a (\check{x}_\alpha)$, but this contradicts the choice of $p$ and $N$.

Finally to see that the mapping is cofinal, let $f \in \bbb$. By the $\aleph_2$-c.c. there is some name $\dot{f}$ so that $\dot{f}_G = f$ and some $X \subseteq \mathbb Q$ of size $\aleph_1$, and hence bounded so that there is a $b$ greater than every $x \in X$ and so that $\dot{f}$ is equivalent to a $\mathbb{D} (\mathbb Q)_b$ name. But then $f \leq^*_\Null f_G^b$ so we're done. 
\end{proof}

Putting together these lemmas, the rest of the proof of Theorem \ref{mainthmhechler} is relatively straightforward. Let me record the details below.
\begin{proof}[Proof of Theorem \ref{mainthmhechler}]
Fix $G$ as in the lemma above. By Lemma \ref{cof} in $V[G]$ there is a cofinal embedding of $\mathbb Q$ into $(\bbb, \leq^*_\Null)$ and so $\mfb(\leq^*_\Null) = \mfb (\mathbb Q)$ and $\mfd (\leq^*_\Null) = \mfd (\mathbb Q)$. By the fact that the forcing is $\sigma$-closed and has the $\aleph_2$-c.c. it follows that in $V[G]$ $\mfb (\leq^*) = \kappa$ and $\mfd (\leq^*) = \lambda$. Finally, assuming that $|\mathbb{Q}| = \lambda$, like in the example given above of $\mathbb Q = \kappa \times [\lambda]^{< \kappa}$, we can apply a nice name counting argument to also get that $2^\mfc = \lambda$. 
\end{proof}

Let me also observe that the proof of this theorem gives slightly more, in fact it gives a weakened higher dimensional version of Hechler's classical theorem on $\leq^*$, see the remark preceding Theorem 2.5 of \cite{BlassHB}.

\begin{corollary}
Assume $\GCH$ and let $\mathbb Q$ be any well-founded partial order so that $\aleph_2 < \mfb (\mathbb Q) \leq \mfd (\mathbb Q)$ with $\mfb (\mathbb Q)$ and $\mfd (\mathbb Q)$ regular. Then it's consistent that $\mathbb Q$ embeds cofinally into $(\bbb, \leq^*_\Null)$.
\end{corollary}

J. Brendle has shown that under $\CH$ $\mfb (\in^*_\Null) = \mfb (\leq^*_\Null)$ so in the models constructed in Theorems \ref{mainthmhechler} and \ref{mainthmhechler2} $\mfb(\in^*_\Null)$ increases over the iteration, see \cite[Main Theorem 1.2]{BS21}. However, these cardinals can be different, again shown in joint work with J. Brendle, in fact this happens in the Laver Model (for $\mathcal I = \Kb$ or $\mathcal I = \Me$ this happens in the Hechler model). These results, and many more appear in the sequel to the present article, \cite{BS21}. See in particular \cite[Theorem 4.8 and Theorem 4.5]{BS21} for proofs of separation of $\mfb(\in^*_\mathcal I)$ and $\mfb(\leq^*_\mathcal I)$ in the Laver and Hechler models respectively. 

\subsection{Generalizing $\mathbb{LOC}$-Forcing}
In the final subsection here I prove the consistency of having all the cardinals in the diagram arbitrarily large. The relevant forcing is a generalization of the $\mathbb{LOC}$-forcing. 

\begin{definition}
The $\Null$-$\mathbb{LOC}$ forcing, denoted $\mathbb{LOC}_\Null$, is the set of all pairs $(p, \mathcal F)$ so that $p:{\rm dom}(p) \subseteq \baire \to \mathcal S$ is a partial function with a Borel graph and a Borel domain which is measure zero and $\mathcal F \subseteq \bbb$ is countable. We let $(p, \mathcal F) \leq (q, \mathcal G)$ if and only if $p \supseteq q$, $\mathcal F \supseteq \mathcal G$ and for all $x \in {\rm dom}(p) \setminus {\rm dom} (q)$ we have that $g(x) \in^* p(x)$ for every $g \in \mathcal G$.
\end{definition}

Using the same template as with $\mathbb{D}_\Null$ it is straightforward to show the following.
\begin{lemma}
$\mathbb{LOC}_\Null$ is $\sigma$-closed, has the $\mfc^+$-c.c. and adds a function $h:\baire \to \mathcal S$ so that for every $f \in \bbb \cap V$ $f\in^*_\Null h$.
\end{lemma}

As a result of this lemma, using the same ideas as before we get immediately.
\begin{theorem}
Let $\kappa > \aleph_2$ be a regular cardinal and let $\mathbb P_\kappa$ be the countable support iteration of $\mathbb{LOC}_\Null$. Then if $G \subseteq \mathbb P_\kappa$ is generic over $V$ in $V[G]$ we have $\mfb (\in^*_\Null) = \kappa = 2^\mfc$.
\end{theorem}

\section{Conclusion and Questions}

The consistency results above barely hint at the possible constellations of the 18 cardinals considered. Many more splits between the cardinals are shown to be consistent in \cite{BS21} and yet there are still many more that remain open.
\begin{question}
How many of the above defined cardinals defined above can be simultaneously different?
\end{question}
Presumably this would involve developing analogues for well known forcing notions on the reals such as Sacks, Laver etc as I have done for Cohen, Hechler and $\mathbb{LOC}$. I leave this project for future research. 

Since writing this article, I have worked jointly on this extensively with J. Brendle and we have computed the values of these cardinals in standard models of $\neg \CH$ such as the Cohen model, the random model, the Sacks model etc. We have shown that many interesting things happen to the $\mfb (R_\mathcal I)$ cardinals. In these models the $\mfd (R_\mathcal I)$ cardinals all stay $\mfc^+$ however in two step iterations where many Cohen subsets are added to $\omega_1$ first (followed by Sacks, Laver etc) the $\mfd(R_\mathcal I)$ cardinals can be manipulated too. See \cite{BS21}.

Finally let me conclude by noting that, as mentioned in the introduction, the framework introduced is very flexible and many other generalizations are possible. For instance, while I have been working with Baire space, a similar study could easily be carried out for any other uncountable Polish space. One particularly interesting possibility, which I leave for future work, is to consider variations on $\mathfrak{a}$ where $[\omega]^\omega$ is replaced by the $\mathcal I$-positive sets of some Polish space and \say{almost disjoint} means that such sets have intersection in $\mathcal I$. A generalization in this spirit for ideals on $\omega$ has been considered in \cite{HST18}.

\section*{Acknowledgements}

This reasearch was supported by a CUNY mathematics fellowship and the author would like to thank the mathematics department at the Graduate Center at CUNY for this. Parts of this work initially appeared in the author's PhD thesis, \cite{Switzer:dissertation} written under the direction of Professors Joel David Hamkins and Gunter Fuchs. I would like to thank them both for all of their help and encouragement, in particular thanks to Professor Hamkins for listening patiently and giving thoughtful advice and enthusiastic encouragement in the early stages of this work. The author would also like to thank J\"{o}rg Brendle for many helpful discussions about this work and the anonymous referee so many helpful comments. Finally the author would like to thank the Austrian Science Fund (FWF) for the generous support through grant number Y1012-N35.

\bibliographystyle{plain}
\bibliography{Logicpaperrefs}

\begin{thebibliography}{10}

\bibitem{Bar1987}
Tomek Bartoszy\'nski.
\newblock Combinatorial aspects of measure and category.
\newblock {\em Fundamenta Mathematicae}, 127(3):225--239, 1987.

\bibitem{BarJu95}
Tomek Bartoszy\'nski and Haim Judah.
\newblock {\em Set Theory: On the Structure of the Real Line}.
\newblock A.K. Peters, Wellsley, MA, 1995.

\bibitem{BlassHB}
Andreas Blass.
\newblock Combinatorial cardinal characteristics of the continuum.
\newblock In Matthew Foreman and Akihiro Kanamori, editors, {\em Handbook of
  Set Theory}, pages 395--489. Springer, Dordrect, 2010.

\bibitem{brendle16}
J\"{o}rg Brendle, Andrew Brooke-Taylor, Sy-David Friedman, and Diana~Carolina
  Montoya.
\newblock Cicho\'{n}'s diagram for uncountable cardinals.
\newblock {\em Israel J. Math.}, 225(2):959--1010, 2018.

\bibitem{BS21}
J\"{o}rg Brendle and Corey~Bacal Switzer.
\newblock Higher dimensional cardinal characteristics for sets of functions ii.
\newblock {\em ArXiv Preprint}, Submitted:23 Pages, 2021.

\bibitem{CuSh95}
James Cummings and Saharon Shelah.
\newblock Cardinal invariants above the continuum.
\newblock {\em Ann. Pure Appl. Logic}, 75(3):251--268, 1995.

\bibitem{HST18}
Karen~Bakke Haga, David Schrittesser, and Asger T\"ornquist.
\newblock Maximal almost disjoint families, determinacy, and forcing.
\newblock {\em ArXiv Preprint}, 2018.

\bibitem{Hechler74}
Stephen~H. Hechler.
\newblock On the existence of certain cofinal subsets of {$^{\omega }\omega $}.
\newblock In {\em Axiomatic set theory ({P}roc. {S}ympos. {P}ure {M}ath.,
  {V}ol. {XIII}, {P}art {II}, {U}niv. {C}alifornia, {L}os {A}ngeles, {C}alif.,
  1967)}, pages 155--173, 1974.

\bibitem{questionsbaire}
Yurii Khomskii, Giorgio Laguzzi, Benedikt L\"{o}we, and Ilya Sharankou.
\newblock Questions on generalised {B}aire spaces.
\newblock {\em MLQ Math. Log. Q.}, 62(4-5):439--456, 2016.

\bibitem{KenST}
Kenneth Kunen.
\newblock {\em Set Theory}.
\newblock College Publications, Studies in Logic, London, 2013.

\bibitem{Miller81}
Arnold~W. Miller.
\newblock Some properties of measure and category.
\newblock {\em Trans. Amer. Math. Soc.}, 266(1):93--114, 1981.

\bibitem{Switzer:dissertation}
Corey~Bacal Switzer.
\newblock {\em Alternative Cicho\'{n} Diagrams and Forcing Axioms Compatible
  with {C}{H}}.
\newblock PhD thesis, The Graduate Center, The City University of New York,
  2020.

\end{thebibliography}

\end{document}